\documentclass[11 pt]{article}

\usepackage{hyperref}
\usepackage{etex}
\usepackage[shortlabels]{enumitem}
\usepackage{amsmath}
\usepackage{amsxtra}
\usepackage{amscd}
\usepackage{amsthm}
\usepackage{adjustbox}
\usepackage{amsfonts}
\usepackage{amssymb}
\usepackage{eucal}
\usepackage[all]{xy}
\usepackage{graphicx}
\usepackage{tikz-cd}
\usepackage{mathrsfs}
\usepackage{subfiles}
\usepackage{mathpazo}
\usepackage[colorinlistoftodos, textsize=tiny]{todonotes}
\usepackage{morefloats}
\usepackage{pdfpages}
\usepackage{thm-restate}
\usepackage[percent]{overpic}
\usepackage[utf8]{inputenc}
\usepackage{epigraph}
\usepackage{csquotes}
\usepackage[margin=1in]{geometry}
\usepackage{adjustbox}
\usepackage{microtype}
\usepackage{stmaryrd}

\usepackage{bm}
\usepackage{verbatim}
\usepackage{stmaryrd}
\usepackage{scalerel}
\usepackage{stackengine}
\stackMath
\newcommand\reallywidehat[1]{%
\savestack{\tmpbox}{\stretchto{%
  \scaleto{%
    \scalerel*[\widthof{\ensuremath{#1}}]{\kern-.6pt\bigwedge\kern-.6pt}%
    {\rule[-\textheight/2]{1ex}{\textheight}}
  }{\textheight}%
}{0.5ex}}%
\stackon[1pt]{#1}{\tmpbox}%
}
\usepackage{mathtools}

\graphicspath{ {images/} }

\RequirePackage{color}
\definecolor{myred}{rgb}{0.75,0,0}
\definecolor{mygreen}{rgb}{0,0.5,0}
\definecolor{myblue}{rgb}{0,0,0.65}

\usepackage{color}

\usepackage{hyperref}
\hypersetup{citecolor=blue}
\usepackage{tikz}
\usetikzlibrary{matrix,arrows,decorations.pathmorphing}

\newtheorem{theorem}{Theorem}[section]
\newtheorem{lemma}[theorem]{Lemma}
\newtheorem{proposition}[theorem]{Proposition}
\newtheorem{conjecture}[theorem]{Conjecture}
\newtheorem{corollary}[theorem]{Corollary}
\newtheorem*{theorem*}{Theorem}
\newtheorem*{proposition*}{Proposition}
\newtheorem*{corollary*}{Corollary}

\theoremstyle{definition}
\newtheorem{defn}[theorem]{Definition}
\newtheorem{assumption}[theorem]{Assumption}
\newtheorem{setup}[theorem]{Setup}

\newtheorem{claim}[theorem]{Claim}

\theoremstyle{remark}
\newtheorem{remark}[theorem]{Remark}

\newcommand\nc{\newcommand}
\nc\on{\operatorname}
\nc\renc{\renewcommand}

\DeclareMathOperator{\acts}{\curvearrowright}

\DeclareMathOperator\GL{GL}

\DeclareMathOperator\End{End}

\DeclareMathOperator\Sh{Sh}
\DeclareMathOperator{\ord}{ord}
\DeclareMathOperator{\Gal}{Gal}
\DeclareMathOperator{\ad}{ad}
\DeclareMathOperator{\der}{der}
\DeclareMathOperator{\uni}{uni}
\DeclareMathOperator{\spl}{split}

\DeclareMathOperator{\imag}{Im}
\DeclareMathOperator{\SC}{sc}
\DeclareMathOperator{\Frob}{Frob}

\DeclareMathOperator{\BB}{BB}

\DeclareMathOperator{\lie}{Lie}

\newcommand{\ShimK}{S_K(G,X)}
\newcommand{\ShimKprime}{S_{K'}(G,X)}

\newcommand{\Shtil}{\widetilde{\Sh}}

\newcommand{\abvar}{\mathcal{A}}

\newcommand{\Fpbar}{\overline{\mathbb{F}}_p}

\newcommand{\Ag}{\mathcal{A}_g}
\newcommand{\Agn}{\Ag[n]}
\newcommand{\Hilbert}{\mathcal{H}}
\newcommand{\Afp}{\mathbb{A}^{f,p}}
\newcommand{\Af}{\mathbb{A}^{f}}

\nc\mf\mathfrak
\nc\mc\mathcal
\nc\mb\mathbb
\nc\msf\mathsf
\nc\mscr\mathscr

\def\fpbar{{\overline{\mb{F}}_p}}
\def\F{{\mb{F}}}

\newcommand{\defeq}{\vcentcolon=}
\newcommand{\Qbar}{\overline{\mathbb{Q}}}

\newcommand{\Fbar}{\overline{F}}
\newcommand{\Shfbar}{\Sh_{\Fbar}}
\newcommand{\Ql}{\mathbb{Q}_{\ell}}

\DeclareMathOperator{\et}{\textrm{\'et}}

\newcommand{\Q}{\mathbb{Q}}

\newcommand{\C}{\mathbb{C}}
\newcommand{\A}{\mathbb{A}}
\newcommand{\Z}{\mathbb{Z}}

\newcommand{\G}{\mathbb{G}}
\newcommand{\W}{\mathbb{W}}

\newcommand{\cH}{\mathcal{H}}

\title{A characteristic $p$ analogue of the Andr\'e--Pink--Zannier conjecture}
\author{Yeuk Hay Joshua Lam, Ananth N. Shankar}
\date{\today}

\begin{document}
\maketitle

\begin{abstract}
We investigate the analogue of the Andr\'e--Pink--Zannier conjecture in characteristic $p$. Precisely, we prove it for ordinary function field-valued points with big monodromy, in Shimura varieties of Hodge type. We also prove an algebraic characteristic $p$ analogue of Hecke-equidistribution (as formulated by Mazur) for Shimura varieties of Hodge type. We prove our main results by a global and local analysis of prime-to-$p$ Hecke correspondences, and by showing that Weyl special points are abundant in positive characterstic. 

\end{abstract}

\section{Introduction}

The Andr\'e--Pink--Zannier conjecture is an important milestone in the theory of unlikely intersections, and is a special case of the more general Zilber--Pink conjecture. Roughly, it posits that, on a  characteristic zero Shimura variety,  every irreducible component of the Zariski-closure of a subset of a  Hecke orbit is \emph{weakly special}\footnote{See \cite{Moonenlinearity} for the definition of weakly special subvarieties}: these are the totally geodesic subvarieties, and include  Shimura subvarieties. The first version of this conjecture was formulated by Andr\'e in \cite[\S X 4.5, p. 216 (Problem 3)]{Andre}, with other versions formulated and studied by Yafaev, Pink, and Zannier. In a recent work \cite{AndrePinkAbelian}, Rodolphe--Yafaev have proven this conjecture for Shimura varieties of abelian type. 

In this article, we begin the investigation of this conjecture in positive characteristic $p$. As we shall see, some care is needed to even formulate a plausible positive characteristic analogue. Firstly,  \emph{Newton strata} of a Shimura variety are invariant under  Hecke correspondences, but are usually not reductions mod $p$ of weakly special subvarieties. So, we restrict ourselves to the ordinary locus of Shimura varieties. However, we still expect that the naive translation with ordinary finite field-valued points, in place of characteristic zero points, is false. Indeed, we expect that an ordinary irreducible component of the example constructed in \cite[Section 4]{AJunlikely} will intersect the Hecke orbit of any ordinary point in a Zariski-dense subset. We note that the generically ordinary components all have maximal geometric monodromy, and therefore are not weakly special. 

Our main result is in the setting of characteristic $p$ function fields. To that end, let $F = \overline{\F_p(t)}$. Let $\Ag$ be the moduli space of principally polarized abelian varieties of dimension $g$, and let $\ShimK \hookrightarrow \Ag$ be a simple\footnote{By simple, we mean that the Shimura variety is not the  product of two positive-dimensional Shimura varieties. Equivalently, the derived group $G^{\der}$ is a $\Q$-simple group.} Shimura variety of Hodge type that embeds in $\Ag$. Let $E(G,X)$ be the reflex field of $\ShimK$. We will assume throughout the paper that $\ShimK$ (and indeed, every Shimura variety we consider) has hyperspecial level at $p$. Let $\frak{p}$ be a prime of $E(G,X)$ above $p$. Then by work of Kisin \cite{Kisinintegral}, the embedding $\ShimK \hookrightarrow \Ag$ extends to a map of integral models over $\frak{p}$. For brevity, we will let the $\Ag$ and $\ShimK$ denote the mod $\frak{p}$ special fibers of $\Ag$ and $\ShimK$ respectively. 

Recall that, for each prime $\ell\neq p$, the \'etale cohomology of the universal family of abelian varieties furnishes representations of the \'etale fundamental group of $\ShimK$; in particular, we have a representation  $\rho_{\ell}: \pi_1(\ShimK\otimes \fpbar)\rightarrow G^{\der}(\Ql)$ for each $\ell$. Here, $G^{\der}$ denotes  the derived subgroup of $G.$

\begin{defn}
    Let $C\subset \ShimK$ be an irreducible curve defined $\Fpbar$. We say that $C$ has big monodromy if the monodromy representation $\rho_{\ell}|_{\pi_1(C)}$ has Zarisky dense monodromy in $G^{\der}(\Q_\ell)$ for some prime $\ell \neq p$. 
    
    We say that $x\in \ShimK(F)$ has big monodromy if $x\in \ShimK(H)$ for a sub-extension $\Fpbar\subset H\subset F$, with $H$ finitely generated over $\fpbar$, and a prime number $\ell\neq p$, such that the representation $\rho_{\ell}|_{G_H}$ has Zariski dense monodromy in $G^{\der}_{\Ql}$. 
\end{defn}

As we explain in \autoref{section: shim-set-up}, this definition is independent of the choice of $\ell$. 
\begin{theorem}\label{thm: intro function field}
    Suppose that $\ShimK$ intersects the ordinary locus of $\Ag$. Let $x\in \ShimK(\overline{\F_p(t)})$ be an ordinary point with big monodromy. Let $\Sigma$ be an infinite subset of the prime-to-$p$ Hecke orbit of $x$. Then, the Zariski closure of $\Sigma$ is a finite union of connected components of $\ShimK$. 
\end{theorem}

We note that this is a result that parallels the ordinary Hecke orbit conjecture as conjectured by Chai--Oort, proved by Chai in \cite{Chai-inventiones} for $\Ag$, and proved by van Hoften for  general  Hodge type Shimura varieties (see also  \cite{MST} for the case of orthogonal Shimura varieties, and see \cite{MarcoPol} for the non-ordinary setting). Indeed, the ordinary Hecke orbit  conjecture is the statement that the entire $\ell$-Hecke orbit of an ordinary $\Fpbar$-point is dense. 
By contrast, Theorem \ref{thm: intro function field} proves the density of an arbitrary infinite subset of the prime-to-$p$ Hecke orbit, with the caveats that the initial point has  big monodromy and that it is a $F$-point. As discussed above, the latter condition is crucial, while we expect that  the former can be removed with extra work.

As a key step towards this result, we prove the following result of a more geometric nature. 

\begin{theorem}\label{thm: intro curve}
Let $\ShimK$ be as above, and let $C\subset \ShimK$ be an irreducible curve with big monodromy that intersects the ordinary locus that is defined over $\F_q$. Let $\tau_i$ be a sequence of non-trivial prime-to-$p$ Hecke correspondences such that $\deg \tau_i \rightarrow \infty$. Let $C_i \subset \tau_i(C)$ be a union of irreducible components. Then, the Zariski closure of $\cup_i C_i$ is a finite union of connected components of $\ShimK$.  

\end{theorem}

The characteristic zero analogue of Theorem \ref{thm: intro function field} (i.e. the Andr\'e--Pink--Zannier conjecture in the setting of $\overline{\Q}$-points with big monodromy) follows from the archimedean equidistribution of Hecke orbits \cite{ClozelOhUllmo}, as observed by Pink in \cite[Section 7]{Pink}. A rough sketch of Pink's argument is as follows. Suppose $x\in \ShimK(\Qbar)$ has big monodromy; given a Hecke correspondence $\tau$ on $\ShimK$, the set $\tau(x)$ breaks up into a bounded number of Galois orbits (where the bound is independent of $\tau$). This implies that the Zariski closure of an infinite subset of $\cup_i\tau_i(x)$ must contain a positive proportion of $\tau_i(x)$ for an infinite sequence of $i$. Equidistribution of Hecke orbits now shows that the closure must be a union of connected components of the Shimura variety.

This argument does not go through in positive characteristic, as there is no proven analogue of equidistribution of Hecke correspondences. In fact, the characteristic zero statement is in terms of the hyperbolic metric on $\ShimK$, which, of course, does not have a positive characteristic analogue. Nevertheless, Mazur has formulated the following algebraic notion of equidistribution:

\begin{defn}[Mazur]\label{defn: intro equidistrbution}
    Let $Y\subset \ShimK$ be a subvariety of positive dimension. The Hecke orbit of $Y$ is said to equidistribute if the minimal degree of a hypersurface containing $\tau_i(Y)$ tends to infinity for any sequence of prime-to-$p$ Hecke correspondences $\tau_i$ that satisfy $\deg \tau_i \rightarrow \infty$.
\end{defn}
Given this definition, it seems reasonable to posit the following:
\begin{conjecture}[Algebraic equidistrbution of Hecke orbits]\label{conj: intro equidistribution}
    Let $Y \subset \ShimK$ be a positive-dimensional subvariety that intersects the ordinary locus. Then the Hecke orbit of $Y$ equidistributes.
\end{conjecture}
 Note that the ordinary assumption is necessary, since otherwise we may simply choose $H$ to be the non-ordinary locus, which is stable under all Hecke correspondences. When $Y$ is contained in a lower-dimensional Newton stratum, the formulation of \autoref{conj: intro equidistribution} should be in terms of the smallest subvariety of $\ShimK$ that is a union of central leaves and contains $Y$. Theorem \ref{thm: intro curve} now implies the following theorem.
\begin{theorem}\label{thm: intro equidistrbution for big monodromy}
    Let $C\subset \ShimK$ be a generically ordinary curve defined over $\Fpbar$ having big monodromy. Then the Hecke orbit of $C$ equidistributes.  
\end{theorem}

Indeed, suppose that $C$ is defined over $\F_q$. We may quickly reduce to the case of hypersurfaces defined over $\Fpbar$ by standard specialization arguments. Now let $H_i/\Fpbar$ be the hypersurface with minimal degree containing $\tau_i(C)$. By intersecting $H_i$ with its Galois conjugates, we may further reduce to the case that the $H_i$ are all defined over a fixed finite field, independent of $i$. If the degrees of $H_i$ were bounded for some subsequence $i$, we may pass to a further subsequence and deduce that the sequence of hypersurfaces $H_i$ is constant, as there are only finitely many subvarieties having bounded degree defined over a fixed finite field. The conclusion now follows from Theorem \ref{thm: intro curve}. 

In future work, we plan to explore Mazur's notion of algebraic equidistribution beyond the cases discussed above. 

\subsection{An application: isogenous Jacobians}

Recall the folklore question posed by Katz--Oort, which asks whether there exist abelian varieties $A$ over $\Qbar$ which are not isogenous to the Jacobian of any curve. More generally (following Poonen), given any subvariety $D\subset \Ag$ and a field $L$, one can ask for the existence of a point $x\in \Ag(L)$ that is not isogenous to any $y\in D(\overline{L})$. The existence of such a point (with $L = \overline{\Q}$) was proved by Chai-Oort and Tsimerman in \cite{ChaiOortJacobian} and \cite{Jacob}, and also independently by \cite{MasserZannier} and \cite{AJ}. Over $\Fpbar$, this expectation depends heavily on the dimension of $D$ (see \cite{AJunlikely}). However, the main theorem of \cite{AJ} proves the existence of such  a function field-valued point. An immediate corollary of our main theorem is the following finiteness result: 
\begin{corollary}\label{cor: intro jacobians}
    Let $D\subsetneq \ShimK$ be any subvariety defined over $F = \overline{\F_p(t)}$. Let $x\in \Sh(F)$ be any ordinary point with big monodromy. Then, there are only finitely many points in the prime-to-$p$ Hecke orbit of $x$ which lie in $D$. In particular, let $A/F$ be any $g$-dimensional abelian variety (with $g\geq 4$) having big monodromy. Then there are only finitely many Jacobians in the prime-to-$p$ isogeny class of $A$. 
\end{corollary}
One could ask whether the strenghtening of this result, in the style of Coleman--Oort, holds true: i.e. whether it is possible to find a Jacobian of dimension $g$ which is not isogenous to any other Jacobian (we thank Olivier de Gaay Fortman for asking us this question); this is a more restrictive version of the question of Katz--Oort. See \cite{Olivier} for results in characteristic-zero which go significantly beyond this specific question. \autoref{cor: intro jacobians} shows that \emph{every} ordinary $g$-dimensional Jacobian over $F$, having big monodromy, will be isogenous to at most finitely many other Jacobians.

\subsection{Outline of proof of Theorem \ref{thm: intro curve}}
Let $V$ be the Zariski closure of the $C_i$'s, where notation is as in Theorem \ref{thm: intro curve}. Our proof of Theorem \ref{thm: intro curve} proceeds in the following steps. 
\begin{enumerate}
    \item  We first reduce to the case that $C$ has maximal prime-to-$p$ monodromy.  Crucial inputs here are the   works of  Cadoret--Hui--Tamagawa and B\"ockle--Gajda--Petersen. 

    \item We then establish degree bounds on Hecke translates of $V$ and use these bounds in conjunction with finiteness arguments to deduce that $V$ is fixed by a Hecke correspondence, $\tau$. This is \autoref{thm: hecke fixed}.

    \item We then prove that $V$ has infinitely many \emph{Weyl special points} (see \autoref{defn: Weyl-special} for this notion, which  was introduced in \cite{ChaiOortJacobian}), including points that are fixed by $\tau$. Our result gives a characteristic $p$ proof of the existence of Weyl special points in Shimura varieties -- 
    this was originally proved by Chai--Oort in \cite{ChaiOortJacobian}. 

    \item  Using crucially a rigidity theorem of Chai (\cite{Chairigidity}), we show that any subvariety fixed by $\tau$, and also  containing  a Weyl special point  fixed by $\tau$, must be Tate-linear. 

    \item We now use arguments of Chai, van Hoften, and D'Addezio--van Hoften (which uses the parabolicity conjecture proved by D'Addezio) to immediately conclude that $V$ equals a union of connected components of $\ShimK$.
\end{enumerate}

\subsection{Prior and related work.}

The first version of the Andr\'e--Pink--Zannier conjecture was formulated by Andr\'e \cite[\S X 4.5, p. 216 (Problem 3)]{Andre}, in the setting where an algebraic curve meets a Hecke orbit in a Zariski dense set of points. Subsequently, various versions of the conjecture were formulated and studied independently by Pink \cite{Pink}, Yafaev \cite{yafaev2000sous}, and Zannier \cite{Pink}.  Using the Pila--Zannier strategy (and the Pila--Wilkie theorem)  special cases were obtained by Orr \cite{orr} and Richard--Yafaev \cite{richard-yafaev-height}, and  as mentioned above, it was very recently proven by Richard--Yafaev \cite{AndrePinkAbelian} for a vast class of Shimura varieties, namely those of abelian type. For more on the history of this problem, we refer the reader to \cite[\S 1.2]{richard-yafaev-height}.

\subsection{Organization of the paper}
In Section 2, we define and recall various notions and reduce to the case of maximal monodromy. In Section 3, we show that $V$ must be stable by a Hecke correspondence. In Section 4, we prove that $V$ contains Weyl special points, that $V$ is Tate-linear, and thereby conclude Theorem \ref{thm: intro curve}. In Section 5, we then prove Theorem \ref{thm: intro function field}.

\subsection{Acknowledgments}
We are grateful to Greg Baldi, Olivier de Gaay Fortman, Barry Mazur, and Salim Tayou for helpful discussions. Throughout the course of this work, J.L. was supported   by a Dirichlet Fellowship and the DFG Walter Benjamin grant \emph{Dynamics on character varieties and Hodge theory} (project number: 541272769). A.S. is partially supported by the NSF CAREER grant DMS 2338942 and a Sloan Research fellowship.

\section{Setup}
Let $\F$ denote either $\F_q$ or $\Fpbar$. Consider the setting of $S$, a geometrically irreducible variety over $\F$, and an abelian scheme $\mathcal{B}/S$. $\mathcal{B}$ induces $\ell$-adic local systems $T_\ell$ on $S$ for $\ell \neq p$, with associated representations $\rho_{\ell}$ of $\pi_1(S)$, and an $F$-crystal $D$ on $S$. We will use the following terminology.
\begin{enumerate}
    \item For a prime $\ell \neq p$, the term $\ell$-adic monodromy of $S$ will refer to the monodromy of $T_{\ell}$. 
    \item The prime-to-$p$ monodromy of $\mathcal{B}$ or $S$ will refer to the monodromy of $\prod_{\ell \neq p} T_{\ell}$.
    \item The term overconvergent $p$-adic monodromy of $\mathcal{B}$ or $S$ will refer to the monodromy of the isocrystal $D\otimes \Q_p$ in the overconvergent category.
    \item The term $p$-adic monodromy of $\mathcal{B}$ or $S$ will refer to the monodromy of $D\otimes \Q_p$ (without any convergence conditions).
    \item The term geometric ($\ell$-adic or prime-to-$p$ or $p$-adic) monodromy of $\mathcal{B}$ or $S$ will refer to the ($\ell$-adic or prime-to-$p$ or $p$-adic) monodromy of $S_{\Fpbar}$. 
\end{enumerate}

\subsection{Shimura varieties}\label{section: shim-set-up}
Now, suppose that $(G,X)$ is a Shimura datum which admits a map of Shimura data to $(\textrm{GS}_{\textrm{p}},S^{\pm})$, such that the following holds:
\noindent\begin{assumption}\label{assumption: shim}
    \begin{enumerate}
\hfill\item The kernel of the map $G\rightarrow \textrm{GS}{\textrm{p}}$ is central in $G$. 
 \item $G^{\der}$ is $\mb{Q}$-simple.
        \item The level structure $K\subset G(\Af)$ is neat and satisfies $K = \prod_\ell K_\ell$.
        \item The level structure at $p$, $K_p$, is a hyperspecial subgroup of $G(\mb{Q}_p)$.
    \end{enumerate}
\end{assumption}

 We choose a Hodge morphism (necessarily a finite map) of $\ShimK$ into $\Ag$ induced by the symplectic $\Q$-representation $W$ of $G$ induced by the map $G\rightarrow \textrm{GS}_{\textrm{p}}$. We let $\W\subset W$ be a self-dual lattice such that $\W \otimes \hat{\Z}$ is stable by $K$.  Let $\abvar^{\uni}/\ShimK$ denote the universal abelian scheme on $\ShimK$ induced by Hodge morphism of $(G,X)$. We will assume without losing generality that there is no subtorus $Z'$ of the center $Z(G)$ such that the conjugacy class of maps from the Deligne Torus factors through the group generated by $G^{\der}$ and $Z'$.

 Let $E = E(G,X)$ be the reflex field. Let $p$ be a prime at which at the level structure is hyperspecial. By work of Kisin, $\ShimK$ has an integral canonical model over $O_{E,v}$, where $v$ is any prime of $E$ dividing $p$. We abuse notation and let $\ShimK$ denote the mod $v$ special fiber of the integral canonical model. At times, it will not be necessary to explicitly keep note of the level structure $K$, in which case we will use the notation $\Sh$ for our mod $v$ Shimura variety.

Following Cadoret \cite{cadoret2017geometric}, we need the following definition in the setting where $G^{\der}$ is not simply connected. 
\begin{defn}
    Let $H$ be a semisimple group. We let $H^{\SC}$ denote its simply connected cover. If $H$ is over $\Q_\ell$, We say that a compact open subgroup of $H(\Q_\ell)$ is \emph{almost hyperspecial} if it is the image of a hyperspecial subgroup of $H^{\SC}(\Q_\ell)$. 

    Given a reductive group $G$, we will let the same symbol $G^{\SC}$ denote $(G^{\der})^{\SC}$.
\end{defn}

\begin{remark}\label{remark: non-sc}
    Shimura varieties of Hodge type (i.e. the map $G \rightarrow \textrm{GS}_{\textrm{p}}$ is injective) where the simple factors of $G$ have type $A,B,C$ always have the property that $G^{\der}$ is simply connected. If $G$ is of type $D$ and the Shimura data is of K3 type (the situation described in \cite[Section 3.4]{Kai-Wen}), then it will also satisfies this property. However, there are Shimura varieties of Hodge type where $G$ has type $D$, but the Shimura datum is not of K3 type (the situation described in \cite[Section 3.3]{Kai-Wen}). In this case, $G^{\der}$ will not be simply connected. For $\ell \gg 1$, the index of an almost hyperspecial subgroup inside a hyperspecial subgroup has order a power of 2, and a set of coset representatives can be chosen to lie in any maximally split maximal torus. We briefly describe how to establish this. Pick a split quadratic form and pick a basis so that the Gram matrix of the form is block anti-diagonal, with the anti-diagonal blocks equaling the identity matrix. Then, a diagonal matrix of the form $(\alpha,1,\hdots 1, \alpha^{-1},1, \hdots 1)$ together with the identity matrix will be a full set of coset-representatives for any choice of non-square unit $\alpha$. Analogous constructions work even for the non-split reductive orthogonal groups over $\Z_\ell$. 
\end{remark}

 Let $T_{\ell}$ denote the family of $\ell$-adic Tate modules on $\Sh$ for $\ell \neq p$ and $T^{(p)} = \prod_{\ell \neq p} T_{\ell}$, and let $D$ denote the relative crystalline cohomology of $\abvar^{\uni}$. Then, the geometric prime-to-$p$ monodromy of $\ShimK$ is the product of the $\ell$-adic monodromy ranging over primes $\ell \neq p$. This follows from Assumption 3 in \ref{assumption: shim}. Further, for $\ell \gg 1$, the geometric $\ell$-adic monodromy is an almost hyperspecial subgroup of $G^{\der}(\Q_\ell)$.

Let $\omega$ denote the Hodge line bundle on $\Sh$. We denote by $\Sh^{\BB}$ the Baily--Borel compactification of $\Sh$. Recall  that $\omega$ extends canonically to a line bundle $\Sh^{\BB}$, which we continue to denote by $\omega$.

Now, let $V$ be a geometrically irreducible subvariety of $\Sh$ defined over $\F_q$, and let $V_0$ denote its smooth locus. By \cite[Theorem 1.2]{Bockle-Gajda-Petersen}, after replacing $V_0$ by a finite cover, we have that the geometric prime-to-$p$ monodromy of $V_0$ equals the product of the geometric $\ell$-adic monodromy where $\ell$ ranges over all primes different from $p$. Suppose that for some prime $\ell \neq p$,  the geometric $\ell$-adic monodromy of $V_0$ has the Zariski closure equal to $G^{\der}_{\Q_{\ell}}$. By compatibility, the same is true for every other prime $\ell' \neq p$.  By \cite[Theorem 1.2]{cadoret2017geometric}, the geometric $\ell$-adic monodromy of $V_0$ equals the geometric $\ell$-adic monodromy of $\Sh$ for all but finitely many primes $\ell$.

\begin{defn}
We say that $V$ as above has \emph{big} geometric monodromy. For a prime $\ell$, if the geometric $\ell$-adic monodromy (and not just its Zariski closure) of $V_0$ equals that of $\Sh$, we say $V$ has \emph{maximal} geometric $\ell$-adic monodromy. If the geometric prime-to-$p$ monodromy of $V_0$ equals the geometric prime-to-$p$ monodromy of $\Sh$, we say that $V$ has \emph{maximal} geometric prime-to-$p$ monodromy.
    
\end{defn}

By replacing the level subgroup $K$ by a finite-index subgroup $K'\subset K$ which we may also assume is neat, we have that the map $V\rightarrow \ShimK$ lifts to a map $V' \rightarrow \ShimKprime$ where $V'\rightarrow V$ is finite \'etale cover, and the image of $V'$ in $\ShimKprime$ has maximal prime-to-$p$ monodromy. 

\subsection{Hecke correspondences}

We now describe the action of prime-to-$p$ Hecke correspondences, with a view towards Galois actions. To that end, let $h \in G(\A^{f,p})$. Let $K' = hKh^{-1} \cap K$. The Shimura variety $\ShimKprime$ admits canonical maps $\pi_1, \pi_2$ to $\ShimK$ and $S_{hKh^{-1}}(G,X)$, induced by the inclusions $K'\subset K$ and $K'\subset hKh^{-1}$ respectively. Moreover, we have a natural isomorphism $\iota: S_{hKh^{-1}}(G,X)\simeq \ShimK$ induced by conjugation by $h^{-1}$; write $\pi_h$ for the composition $\iota \circ \pi_2$. 
  The Hecke correspondence $\tau_h^K$ is defined as
\[
\begin{tikzcd}[column sep=small]
& S_{K'}(G, X) \arrow[dl, "\pi_1"'] \arrow[dr, "\pi_h= \iota \circ \pi_2"] & \\
  S_K(G, X)  &                         & S_K(G, X)
\end{tikzcd}
\]
When the level is clear, we sometimes omit the superscript $K$ and simply refer to this correspondence as $\tau_h$. We will use the following convention: 

\begin{defn}
    Let $\tau = \tau_h$ be a Hecke correspondence as above. We let $\tau^\vee$ denote the Hecke correspondence $\tau_{h^{-1}}$. 
\end{defn}

We will need the following result, which is well known. 
\begin{proposition}\label{prop: hecke same degree}
Consider the setting above. Then, the degree of the map $\pi_1$ equals the degree of the map $\pi_h$.
\end{proposition}
\begin{proof}
    We include an argument for completeness. We must prove that the index of $K' = K \cap hKh^{-1}$ in $K$ equals its index in $hKh^{-1}$. We may work prime-by-prime, so suppose that $h \in G(\Q_\ell)$ for some prime $\ell$. 
    
    As $G$ is reductive, we have that $G(\Q_\ell)$ is a unimodular group, and so a left-invariant measure will also be right-invariant. Pick such a Haar measure. Then, the index of $K'$ in $K$ is simply the ratio of their volumes. Similarly, the index of $K'$ in $hKh^{-1}$ is again the ratio of their volumes. The volume of $K$ equals that of $hKh^{-1}$ as the measure is a Haar measure, and the result holds. 
\end{proof}

The following is now well defined: 

\begin{defn}\label{def: degree of hecke}
    The degree of a Hecke correspondence $\tau_h$ is defined to be the degree of the maps $\pi_1, \pi_h$. 
\end{defn}

Let $h\in G(\Afp)$. As $\tau_h$ = $\tau_{hk}$ for $k\in K$, we may assume that $h$ is semisimple. Not all elements $h$ induce ``non-trivial'' correspondences. For instance, if $h\in K$, then $\tau_h$ is just the identity. Similarly, if $K \subset K''$ is a normal subgroup where $K''$ is also a compact open subgroup and if $h\in K''$, then $\tau_h$ will be single-valued and will be a Deck transformation for the Galois \'etale cover $\ShimK \rightarrow S_{K''}(G,X)$, and will therefore be a periodic function. In order to ensure that the Hecke correspondences we work with are not trivial in this sense, we make the following definition. 

\begin{defn}\label{def: non-trivial Hecke}
     Let $h_{\ell}\in G(\Q_\ell)$ be a semisimple element. We say that $\tau_{h_\ell}$ is a non-trivial Hecke correspondence if $h^{\ad}_{\ell} \in T(\Q_{\ell}) \setminus K_T$, where $h^{\ad}_\ell$ is the image of $h$ in $G^{\ad}$, $T\subset G^{\ad}(\Q_\ell)$ is a maximal torus containing $h^{\ad}_{\ell}$, and $K_T$ is the maximal compact subgroup of $T$. 

    Let $h = (h_{\ell})_{\ell} \in G(\Afp)$. We say that $\tau_h$ is a non-trivial Hecke correspondence if $\tau_{h_\ell}$ is non-trivial for at least one prime $\ell$.
\end{defn}

For the rest of this paper, we will work exclusively with non-trivial Hecke correspondences. Note that for a non-trivial Hecke correspondence $\tau_h$, the degree of $\tau_{h^n}$ goes to infinity with $n$.

\subsubsection{Connected components}
We work over an algebraically closed field of characteristic $p$. We remark that everything in this paragraph is well known to the experts. We introduce the following notation. 

\begin{defn}
    For $K \subset G(\Af)$, set $K^{\der} := K \cap G^{\der}(\Af)$, $K^Z := K\cap Z(G)(\A^f)$, and $K^{\SC} : = K \cap \textrm{Im}(G^{\SC}(\Af))$.
    For $K_\ell \subset G(\Q_\ell)$, set $K^{\der}_\ell := K_\ell \cap G^{\der}(\Q_\ell)$, $K^Z_\ell = K \cap Z(G)(\Q_\ell)$ and $K^{\SC}_\ell := K_\ell \cap \textrm{Im}(G^{\SC}(\Q_\ell))$. 
\end{defn}

We start with the following proposition. We will not suppress level structure in our notation. 

\begin{proposition}\label{prop: geometric connected components}
Suppose $K = \prod K_\ell \subset G(\Af)$, hyperspecial at $p$, and for every $\ell$ we either have $K_\ell = K_\ell^{\SC} \cdot K_\ell^Z$ or $K_{\ell}$ is hyperspecial. Then, for any $h\in G(\Afp)$, $\ShimKprime \xrightarrow{\pi} \ShimK$ restricted to a geometric connected component of $\ShimK$ is an irreducible finite \'etale cover. Here, $K' = K \cap hK h^{-1}$.  
\end{proposition}
\begin{proof}
Let $S_0$ be a geometric connected component of $\ShimK$. The action of the geometric fundamental group of $S_0$ on the fibers of $\pi$ restricted to $S_0$ is naturally identified with the action of $K^{\SC}$ on the coset space $K/K'$. Therefore, it suffices to prove that this action is transitive. The fact that $K = \prod K_\ell$ implies that $hKh^{-1} = \prod_\ell h K_\ell h^{-1}$, and therefore that $K'$ also satisfies $K' = \prod_\ell K'_\ell$. So, it suffices to prove the claim prime-by-prime. 
\begin{itemize}
    \item Suppose  $K_\ell = K_\ell^{\SC} \cdot K_\ell^Z$. Let $\alpha \in K_\ell$. We have that $\alpha = \beta \cdot z$ with $\beta \in K^{\SC}_\ell$ and $z \in K^Z_\ell$. It suffices to prove  the equality of cosets $\beta K' = \alpha K'$, i.e. that $z = \beta^{-1}\alpha \in K'$, i.e. it suffices to prove that $z\in hKh^{-1}$. This follows from the fact that $z$ is central. 

    \item Suppose that $K_\ell$ is hyperspecial, and $K^{\SC}_\ell = K^{\der}_\ell$. We will abuse notation and let $G$ also denote a reductive model over $\Z_\ell$. Given a maximally  split maximal torus $T\subset G$ defined over $\Z_\ell$, we have that the map $T(\Z_\ell) \rightarrow (G/G^{\der})(\Z_\ell)$ is surjective. 
    It follows that there are finitely many elements $ t_1, \hdots, t_n \in T(\Z_\ell)$ such that $K_\ell = \bigcup_i  K^{\der}_\ell K^Z_\ell t_i $. Note that, as the level structure at $\ell$ is hyperspecial, by the refined Cartan decomposition (see e.g. \cite[p.51, \S 3.3.3]{tits-reductive-groups}) we may pick $h \in T(\Q_\ell)$. Now, consider the coset $\alpha K'$, where $\alpha \in K$. We may write $\alpha = \beta \cdot z \cdot t_i$. As in the first case, it suffices to show the equality of cosets $\beta K' = \alpha K'$, i.e. that $z\cdot t_i \in hKh^{-1}$. But this follows because $t_i$ and $h$ commute as they are both elements of $T$. 

    \item Suppose that $K_\ell$ is hyperspecial, and $K^{\SC}_\ell \neq K^{\der}_\ell$ is an almost hyperspecial (but not hyperspecial) subgroup of $G^{\der}(\Q_\ell)$. Then, we have that $G$ must be an even orthogonal group. An argument identical to the one above would go through if a set of coset representatives of $K^{\SC}_\ell$ in $K^{\der}_{\ell}$ could be found in $T$, where $T$ is as above. But this is true, as mentioned in Remark \ref{remark: non-sc}.
\end{itemize}

\end{proof}

We have the following consequence. 

\begin{corollary}\label{cor: can reduce level structure to get hecke irreducible}
Consider the Shimura variety $\ShimK$. At the cost of replacing $K$ by a finite index subgroup, we have that the correspondence variety associated to any Hecke correspondence restricted to a geometrically connected component of $\ShimK$ is irreducible. 

Further, let $Y\subset \ShimK$ be a geometrically irreducible subvavriety having maximal geometric prime-to-$p$ monodromy, and let $h\in G(\Afp)$. Then $\tau^K_h(Y)$ is geometrically irreducible. 
\end{corollary}
\begin{proof}
    For the first claim, it suffices to show that every $K$ contains a compact open subgroup satisfying the hypotheses of the above proposition. At the cost of replacing $K$ by a finite index subgroup, we may assume that $K = \prod_\ell K_\ell$. Further, $K_\ell$ will be hyperspecial for $\ell \gg 1$. For the remaining $\ell$, we may replace $K_\ell$ by $K_\ell^{\der} \cdot K_\ell^Z$.

The second claim follows using the first claim in conjunction with the argument in \cite[Proposition 5.2]{AJ}.
\end{proof}


We now turn to the question of connected components of $\ShimK$. Let $T = G/G^{\der}$. The set of connected components of $\ShimK$ are  in natural bijection with the points  of the zero-dimensional Shimura variety associated with $T$ (with level structure induced by $K$) (see \cite[Section 5]{Milneintro}), which we denote by $\Sh_T$. Let $h\in G(\Afp)$ be some element, and let $h_T$ be its image in $T(\Afp)$. The Hecke correspondence $\tau_h$ induces a correspondence on the set of geometric connected components of $\ShimK$. This correspondence is the same as the correspondence induced on $\Sh_T$ by $\tau_{h_T}$. However, as $T$ is abelian, $h_T$ is a morphism. Therefore, we have the following proposition. 

\begin{proposition}\label{prop: heck connected}
    Let $S_0 \subset \ShimK$ be a connected component over an algebraically closed field, and let $h\in G(\Afp)$. Then, there is a unique connected component $S_1$ of $\ShimK$ such that $\tau_h(S_0) = S_1$. Further, if $h \in G^{\der}$, then $S_1 = S_0$. 
\end{proposition}

For the purposes of the main theorem, it suffices to work with Hecke correspondences that preserve connected components, and so we will assume implicitly  that this is the case for the rest of the paper.




\subsubsection{Reduction to the case of maximal monodromy}
We continue to  work over an  algebraically closed field of characteristic $p$. We will now reduce Theorem \ref{thm: intro curve} to proving the following statement.
\begin{theorem}\label{thm:maximmonodromy}
Let $C \subset \ShimK$ denote a generically ordinary curve with maximal geometric monodromy. Let $\tau_i$ denote a sequence of prime-to-$p$ Hecke correspondences with increasing degrees, and let $C_i = \tau_i(C)$. Then the Zariski closure of the set $\bigcup_i C_i$ is a union of connected components of $\ShimK$.
\end{theorem}

The proposition below follows directly from the definitions.
\begin{proposition}\label{prop: break-up-hecke}
Let $Y\subset \ShimK$ be a geometrically irreducible subvariety with big $\ell$-adic monodromy $K'\subset K_{\ell}$. Let $Y' \subset \ShimKprime$ be an irreducible component of the fiber product $Y\times_{\ShimK} \ShimKprime$. 
Let $h\in G(\Q_\ell)$, with $\tau^K_h$ the Hecke correspondence that it induces on $\ShimK$, and let $Z \subset \tau_h(Y)$. Then, there exist finitely many $h_i \in G(\Q_{\ell})$ such that $Z' = \bigsqcup_i \tau^{K^{\prime}}_{h_i}(Y')$ maps to $Z$.
\end{proposition}

We now deduce Theorem \ref{thm: intro curve} from Theorem \ref{thm:maximmonodromy}.
\begin{proof}
There exists a level $K'\subset K$ with associated covering map  $\pi: \ShimKprime \rightarrow \ShimK$ and a map $h: C\rightarrow \ShimKprime$ such that the composition $\pi \circ h$ recovers the original $C$, and such that the image $\tilde{C}=h(C)$  has maximal prime-to-$p$ monodromy: indeed, for each $\ell$, we may simply take $K'_{\ell}$ to be the $\ell$-adic monodromy of $h$. Let $C_i \subset \tau_i(C)$ be some subvariety (it will be a union of irreducible curves). Now each Hecke correspondence $\tau_i$ on $\ShimK$ breaks up as a union over $j$ of $\tau'_{i,j}$ on $\ShimKprime$. By \autoref{prop: break-up-hecke}, we have that each irreducible component of $C'$ is the image of some  $\tau'_{i,j}(\tilde{C})$. As the map $\ShimKprime \rightarrow \ShimK$ is finite \'etale, to show that the $C_i$ are Zariski dense, it suffices to show that the Zariski closure of $\{\tau'_{i,j}(\tilde{C})\}$ is a union of connected components of $\ShimK'$. But this follows from Theorem \ref{thm:maximmonodromy} for the Shimura variety $\ShimKprime$.  
\end{proof}

\subsubsection{Galois-theoretic properties}
We will end this section  with a description of Galois-theoretic properties of the maps and correspondences in play; we keep the notation as in \autoref{section: shim-set-up}. To analyze the Galois action on points, we drop the assumption that we are working over an algebraically closed field. To that end, let $R$ be a characteristic $p$ field. Recall that we have fixed an element $h \in G(\Afp)$, and that $K' = K \cap hKh^{-1}$. Every point $x\in \ShimK(R)$ is equipped with a $K_{\ell}$-equivalence class of isomorphisms $t_x: \W_\ell \rightarrow T_\ell(\abvar_x)$ where $K_\ell$ acts by pre-composition. Let $g_1 = 1, g_2, \hdots, g_n$ be a set of left coset representatives for $K'_\ell$ in $K_\ell$. Then, the points in the inverse image $\pi^{-1}(x)$ are in bijection with the $K'$-equivalence classes $t_i = t_x \circ g_i: \W_\ell \rightarrow T_\ell(\abvar_x)$. The action of $\sigma \in \pi_{1,et}(R)$ on $\pi^{-1}(x)$ is by post-composing $t_i$ with $\sigma$. The element $x_i \in \pi^{-1}(x)$ corresponding to $t_i: \W_\ell \rightarrow T_\ell(\abvar_x)$ is defined over $R$ if and only if $t_i^{-1}\circ \sigma \circ t_i \in K'$ for every $\sigma \in \pi_{1,et}(R)$. As $x \in \ShimK(R)$, we already have $t_i^{-1}\circ \sigma \circ t_i \in K$. Therefore, the condition reduces to $t_i^{-1}\circ \sigma \circ t_i \in hKh^{-1}$, i.e. $h^{-1}t_i^{-1}\circ \sigma \circ t_i h \in K$. This leads to the following corollary. 

\begin{corollary}\label{cor: heckedescription}
    Let the notation be as above. The set $\tau_h(x)$ has an $R$-rational point if and only if for each $\sigma \in \pi_{1,et}(R)$, there is an element $g\in K$ such that $h^{-1} g^{-1} t^{-1} \sigma t g h \in K$.
\end{corollary}
\begin{proof}
    The forward implication follows directly from the above discussion. Indeed, if $y_i = \pi_h(x_i)$ is the fixed point, then we pick $g = g_i$. 

    For the reverse implication, suppose that $g$ is in the coset $g_i K'$. By the discussion above, it suffices to show that $t_i^{-1} \sigma t_i \in hKh^{-1}$, where $t_i = t\circ g_i$. We have that $k^{-1}t_i^{-1}\sigma t_i k \in hKh^{-1}$, where $g = g_i k$. The result now follows, as $k \in K' \subset  hKh^{-1}$. 
\end{proof}

We now specialize to the case when $R = \F_q$. We introduce the following definition.

\begin{defn}
Given $x \in \ShimK(\F_q)$, we let $\Frob$ denote Frobenius element in $\pi_{1,\et}(x)$. Note that $\Frob$ acts on $T_\ell(\abvar_x)$, and therefore the level structure gives an element $t^{-1} \circ \Frob\circ  t$ in $K_{\ell}$ well-defined upto conjugation. We let $\phi_x$ denote such an element.
\end{defn}
 We have the following key result. 

\begin{theorem}\label{thm: Heckefixed criterion}
    Let $h \in G(\Q_{\ell})$. Then there exists an open subset $U \subset K$ stable under $K$-conjugation, such that if $x\in \ShimK(\F_q)$ satisfies $\phi_x \in U$, then $\tau_{h^n}(x)$ contains an $\F_q$-rational point for every $n$. 
\end{theorem}
\begin{proof}
    Firstly, it suffices to work with $h$ semisimple. Let $Z\subset G$ denote a $\Q_\ell$-maximal torus of $G$ containing $h$, and let $Z_0 = Z(\Ql) \cap K_\ell$. Let $Z_0^r$ denote the subset of regular elements -- we note that $Z_0^r$ is open in $Z_0$. We make the following claim: 

    \begin{claim}
        The set $U = \{ \alpha Z_0^{r} \alpha^{-1}: \alpha \in K_\ell\}$ is an open subset of $K_\ell$.
    \end{claim}
    \begin{proof}[Proof of Claim]

        It suffices to prove that the set $\{\alpha Z_0^r \alpha^{-1}: \alpha \in K'_\ell\}$ is open where $K'_\ell$ is a suitable open subgroup of $K_\ell$. The claim essentially follows from the fact that the fibers of the  map $G/Z \times Z^{r} \rightarrow G$ are finite, but we give a sketch of the argument for completeness.
        Let $f: G\times Z^r \rightarrow G$ be the map $f(g,z) = gzg^{-1}$. It suffices to prove that $d_f|_{(1,z)}:  \lie G \oplus T_z Z \rightarrow T_z G$ is surjective at every point of the form $(1,z)$. 
        The map $d_f|_{(1,z)}$ restricted to ${T_z Z}$ is evidently just the natural inclusion. Therefore, it suffices to show that $d_f(\textrm{Lie }G) \cap T_z Z = \{0\}$ and that the kernel of $d_f$ restricted to $\textrm{Lie}G$ equals $\mathrm{Lie} Z$. Let $g$ be the map obtained by composing $f$ with translation by $z^{-1}$ on the right. Proving the above statement for $d_f$ reduces to proving that $d_g|_{\mathrm{Lie} G}: \textrm{Lie} G \rightarrow \textrm{Lie} G$ satisfies $\ker d_g$ is $\textrm{Lie} Z$ and $\textrm{Im} d_g \cap \textrm{Lie} Z$ is trivial. For $X\in \textrm{Lie} G$, we have that $d_g(X) = X - \textrm{Ad}_z(X)$. Write $X= X_Z + \sum_r X_r$, where $X_Z \in \textrm{Lie} Z$ and each $X_r$ lies in a   root-space of $Z$ with associated co-character $\lambda_r$. Then, $\textrm{Ad}_z(X) =  X_Z + \sum_r \lambda_r(z) X_r$. As $z$ is regular, we have that $\lambda_r(z) \neq 1$ for every $r$, and the claim now follows.

    \end{proof}

The set $U$ does what is required, as we now demonstrate. Indeed, suppose $x$ is such that $\phi_x = t^{-1} \Frob t \in U$, where $t$ is a representative of the $K_\ell$-equivalence class associated to $x$. Therefore, there is an element $g\in K$ such that $g^{-1}t^{-1}\Frob t g \in Z_0^r$. But $h$, and therefore every power of $h$, commutes with $Z_0$. The result now follows by applying  \autoref{cor: heckedescription}.
    
\end{proof}

\section{Stability under Hecke}
\begin{setup}\label{setup: curve-hecke-closure} We use the notation of \autoref{section: shim-set-up}, so that $(G, X)$ is a Shimura datum satisfying \autoref{assumption: shim}.
 Let $K\subset G(\Af)$ be a compact open satisfying the assumptions of Proposition \ref{prop: heck connected}, with associated Shimura variety $\ShimK$ over $\mb{F}_q$. Let $C \subset \ShimK$ be a smooth curve over $\F_q$ with maximal geometric monodromy. Let $\tau_i$ be a sequence of prime-to-$p$ Hecke correspondences whose degrees go to infinity. Let $V$ be an irreducible component defined over $\Fpbar$ of the Zariski closure of $\{ C_i :=\tau_i(C)\}_{i\geq 0}$. By re-indexing if necessary, we assume $V$ contains $C_i$ for every $i\geq 1$. We let $V_{i,0} \subset \tau^\vee_i(V)$ denote the irreducible component that contains $C$.  
\end{setup}


The main results for the section are the following. 
\begin{theorem}\label{thm: hecke fixed}
Let $V$ be as in the above setup. Then there exists a sequence of distinct prime-to-$p$ Hecke correspondences $\tau'_i$ that satisfy $V\subset \tau'_i(V)$.
\end{theorem}
\begin{corollary}\label{cor: heckefixedsubvariety}
    There is some $i$, and a sequence of distinct prime-to-$p$ Hecke correspondences $\tau'_j$ such that $\tau'_j(V_{i,0}) = V_{i,0}$.
\end{corollary}

\subsection{Degrees of subvarieties}
In this section, we setup notation regarding degrees of subvarieties, and collect together some well known results. 

\begin{defn}
Let $\omega$ denote the determinant of the Hodge bundle on $\Sh$.    We denote by $\Sh^{\BB}$ the  Baily--Borel compactification (often also referred to as a minimal compactification) of $\Sh$: this is the projective variety obtained by taking closure of $\Sh$ under the embedding into projective space given $\omega$
    
    We refer the reader to \cite[\S 5.2]{keerthi} for more details. 
\end{defn}

\begin{proposition}\label{prop: pullbackhodge}
For level subgroups $K'\subset K$ which are both hyperspecial at $p$, with corresponding map $\pi: \Shtil\rightarrow \Sh$, there is a unique finite map  $\pi^{\BB}: \Shtil^{\BB}\rightarrow \Sh^{\BB}$ extending $\pi$. Moreover, ${\pi^{BB}}^*\omega \simeq \omega$. 

\end{proposition}
\begin{defn}
    Let $K$ be a level subgroup as in Assumption \ref{assumption: shim}, with associated Shimura variety $\Sh$. Let $X\subset \Sh$ be an irreducible subvariety. Let $n$ denote an integer such that $\omega^{\otimes n}$ is very ample, and let $\Sh^{\BB} \rightarrow \mathbb{P}^N$
    denote a 
projective embedding defined by $\omega^{\otimes n}$. We define \emph{degree} of $X$ as $\frac{1}{n^{\dim X}} \deg \bar{X}$, where $\bar{X}$ is the closure of $X$ in $\Sh^{BB}$, and the degree is with respect to the projective embedding in $\mathbb{P}^N$. Note that this is a positive rational number. 
\end{defn}

\begin{proposition}\label{prop:degree_fec}
    Let  $K'\subset K$ be level subgroups corresponding to $\pi: \Shtil\rightarrow \Sh$. Let   $X\subset \Sh$ be a subvariety and $Y$ an irreducible component of $\pi^{-1}(X)$. Then $\deg(Y)=\deg(\pi|_{Y})\deg(X)$. 
\end{proposition}
\begin{proof}
We pick an integer $n$ such that $\omega^{\otimes n}$ is very ample  on both $\Sh$ and $\Shtil$. Suppose that $\dim X = d$. We may choose $d$ hyperplane sections $H_1, \hdots, H_d$ of $\omega^{\otimes n}$ in $\Sh^{\BB}$ such that $H_1 \cap  \hdots \cap H_d $ intersects $\bar{X}$ transversely and the intersection is supported away from the boundary $\bar{X} \setminus X$. With this setup, the degree of $X$ is just $\frac{1}{n^d} \# X\cap H_1\cap  \hdots \cap H_d$. 

By Proposition \ref{prop: pullbackhodge}, the $H'_i = {\pi^{\BB}}^{-1}H_i$ are hyperplane sections of $\omega^{\otimes n}$ in $\Shtil^{\BB}$. Further, as the map $\pi$ is finite \'etale and the intersection $\bar{X}\cap H_1\cap  \hdots \cap H_n$ is supported in $\Sh$, we have that the intersection $\bar{Y}\cap H'_1 \cap \hdots \cap H'_d$ is transverse and supported in $\Shtil$. Therefore, the degree of $Y$ is just $\frac{1}{n^d} \# Y \cap H'_1 \cap \hdots \cap H'_d$. 

The proposition now follows from the fact that $Y \cap H'_1 \cap \hdots \cap H'_d = \pi^{-1} (X \cap H_1 \cap \hdots \cap H_d) \cap Y$ which  has cardinality $\deg(\pi|_Y) \cdot \# X \cap H_1 \cap \hdots\cap  H_d $, as $\pi|_{Y}$ is finite \'etale. 

\end{proof}

We require the following results detailing how degrees interact with Hecke correspondences.

\begin{lemma}\label{lemma:degree_cover}
Suppose that $X \subset \Sh$ is a subvariety with big monodromy. Then there is a constant $\kappa >0$, depending only on the index of the monodromy of $X$,  such that the following holds. For $K'\subset  K$ any finite index subgroup,  let  $\Shtil$ be the associated Shimura variety and  $\pi: \Shtil \rightarrow \Sh$  the canonical map, and  $\tilde{X} \subset \Shtil$  any irreducible component of $\pi^{-1}X$. Then the degree of $\pi|_{\tilde{X}}$ satisfies $\kappa \deg \pi \leq \deg \pi|_{\tilde{X}} \leq \deg \pi$.

\end{lemma}
\begin{proof}
    We have that $\pi$ is a finite \'etale map. For brevity, we let $n = \deg \pi$. The second inequality is trivial. 
    
    For the first inequality, let $x\in X$ be some geometric point, and let $\{ x_1, \hdots, x_n\}$ be the fiber of $\pi$ above $x$. There is a transitive action of $K$ on $\{x_1 \hdots x_n \}$ with stabilizer (conjugates of) $K' \subset K$. Suppose that $x_1 \in \tilde{X}$. Then, we have that $\pi_1(X,x) \cdot x_1 \subset K\cdot x_1 = \{x_1, \hdots, x_n \}$ is contained in $\tilde{X}$, and $\deg \pi|_{\tilde{X}} = \# \pi_1(X,x)\cdot x_1$. We may now choose $\kappa = [K:\pi_1(X,x)]^{-1}$, and the lemma follows.   
\end{proof}

\subsection{Hecke stable subvarieties}

We let $V_{i,0}$ denote the irreducible component of $\tau_i(V)$ that contains $C$. We have the following key proposition.

\begin{proposition}\label{prop:key}
   \begin{enumerate}
    \item $V_{i,0}$ has maximal (geometric) monodromy. Further, the field of definition of $V_{i,0}$ has degree  bounded independently of $i$.

    \item $\deg(V_{i,0})$ is bounded independently of $i$;
   
    \end{enumerate} 
\end{proposition}
\begin{proof}
By definition, $V_{i,0}$ is the irreducible component of $\tau_{i}(V)$ containing $C$, and therefore has maximal geometric monodromy since $C$ does. 
The number of irreducible components of $\tau_i(V)$ is bounded independently of $i$ as $V$ has big monodromy. It follows that every irreducible component of $\tau_i(V)$ is defined over a field of degree bounded independently of $i$, whence follows the first part. 

 Abusing notation, we suppose $\tau_i$ is associated with some  element $h\in G(\Afp)$, and  let $\Shtil$ be the correspondence variety for $\tau_i$. Since  $\tau_i(V)$ contains $V_{i,0}$,  there exists an irreducible subvariety  $W_{i,0}\subset \Shtil$  such that 
\[
\pi_1(W_{i,0})=V, \ \pi_2(W_{i,0})=V_{i,0}.
\]
For brevity, let $n_1, n_2$ denote the degrees $\deg(\pi_1|_{W_{i,0}}), \ \deg(\pi_2|_{W_{i,0}})$, respectively. By \autoref{prop:degree_fec}, we have 
\begin{equation}\label{eqn:degrees}
    n_1\deg(V)  = \deg(W_{i,0})= n_2\deg(V_{i,0}).
\end{equation}

On the other hand, by \autoref{lemma:degree_cover} applied to $V_{i,0}$, there exists $\kappa>0$, independent of $i$, such that 
\[\kappa \deg(\tau_i)<n_2, \ n_1 \leq \deg(\tau_i).\]
Combining with \eqref{eqn:degrees}, we deduce 
$\kappa \deg(\tau_i) \deg(V_{i,0}) < \deg(\tau_i)\deg(V)$,
as required.

\end{proof}

\begin{proof}[Proof of \autoref{thm: hecke fixed}]
    Let $\bar{V}_{i,0}\subset \Sh^{\BB}$ denote the closure of $V_{i,0}$. By \autoref{prop:key}, the degree and field of definition of $\bar{V}_{i,0}$ are both bounded independently of $i$. Note that for each $N$, there are only finitely many $\mb{F}_q$-subvarieties of $\mb{P}^N$ with bounded degree (see for example  \cite[Exercise 3.28]{kollar2013rational}).

By the pigeonhole principle, passing to a subsequence,  we may assume  that  the $\bar{V}_{i,0}$ are the same variety for all $i$. Letting $\sigma_i\defeq \tau_1^{\vee}\tau_i$, we obtain that $V \subset \sigma_i(V)$. The theorem follows from Lemma \ref{lem:productoftaunoidentity} below.


    
\end{proof}
The following result must be well known, but we include a proof for completeness.
\begin{lemma}\label{lem:productoftaunoidentity}
    Let $\tau_1$ be a fixed Hecke correspondence and let $\tau_i$ be a sequence of Hecke correspondences with degrees that go to infinity. Then the degree of the minimal irreducible Hecke correspondence in $\sigma_i \defeq \tau_1 \tau_i$ grows with $i$; more precisely, it is bounded below by $\frac{\deg \tau_i}{\deg \tau}$.
\end{lemma}
\begin{proof}
Suppose that $\tau_1 = \tau_h$ and $\tau_i = \tau_g$, where $g,h \in G(\Afp)$. Write $KhK = \bigsqcup_{a = 1}^n K h_a$ and $KgK = \bigsqcup_{b=1}^m K g_m$. We have that $n$ and $m$ are the degrees of $\tau$ and $\tau_i$ respectively. Then, the irreducible components of $\sigma_i$ each have the form $\tau_{h_a g_b}$. The degree of this correspondence equals $[K: K \cap c_{a,b}Kc_{a,b}^{-1}]$ where $c_{a,b} = h_ag_b$. We have 

\begin{align*}[K: K \cap c_{a,b}Kc_{a,b}^{-1}] &\geq [K \cap h_a K h_a^{-1}: K \cap h_a K h_a^{-1} \cap c_{a,b}Kc_{a,b}^{-1}] \\
&= \frac{[h_aKh_a^{-1} :K \cap h_a K h_a^{-1} \cap c_{a,b}Kc_{a,b}^{-1}]}{[h_aKh_a^{-1} : K \cap h_a K h_a^{-1}]} \\
&\geq \frac{[h_aKh_a^{-1}: h_aKh_a^{-1} \cap c_{a,b}Kc_{a,b}^{-1}]}{[h_aKh_a^{-1} : K \cap h_a K h_a^{-1}]}\\
&= \frac{[K: K \cap g_bKg_b^{-1}]}{[K: K \cap h_a^{-1} K h_a]} \\
&= \frac{\deg \tau_i}{\deg \tau},
\end{align*}
as required.
 
\end{proof}

We finish this section by proving \autoref{cor: heckefixedsubvariety}.
\begin{proof}
Let $W_{i,0}$ be as in the proof of Proposition \ref{prop:key}. We have that $W_{i,0} \subset \pi_1^{-1} V_{i,0}$ where $\pi_1$ is the first projection map for the level cover induced by $\tau^\vee$. As $V_{i,0}$ has maximal monodromy at $\ell$, we see that the inclusion must be an equality. Therefore, the infinite sequence of Hecke correspondences $\tau_i^{\vee}$ satisfies $\tau_i^\vee (V_{i,0}) = V$. The same argument as in the last paragraph of the proof of Theorem \ref{thm: hecke fixed} yields the corollary.
\end{proof}


\section{Tate-linearity and proof of \autoref{thm: intro curve}}

We keep the notation as in \autoref{setup: curve-hecke-closure}. We fix the level $K$ throughout and thereore, for brevity, we write simply $\Sh$ for $\ShimK$ and $\Sh_W$ for the model over $W$. Recall that, by \autoref{cor: heckefixedsubvariety}, we may assume $V_{i,0} = \tau_i(V)$ for an infinite sequence of Hecke correspondences $\tau_i$. Proving that $V_{i,0}$ is an entire connected component of $\Sh$ implies \autoref{thm:maximmonodromy} (and therefore, \autoref{thm: intro curve}) and so we may assume that $V$ has maximal geometric monodromy and that there is an infinite sequence of Hecke correspondences $\tau_i$ such that $\tau_i(V) = V$. In fact, we may assume that $\tau_1 = \tau_h$ and $\tau_n = \tau_{h^n}$ for $h\in G(\Afp)$. The main technical theorem that we prove in this section is: 
\begin{theorem}\label{thm:Tate-linear}
    The subvariety $V$ is Tate linear. 
\end{theorem}

We will first prove the following result:  
\begin{proposition}\label{Heckefixedpoint}
With notation as in \autoref{thm: Heckefixed criterion}, let $x$ be an $\F_q$-rational point in the smooth locus of $V$ that satisfies $\phi_x \in U$. Then there exists a non-trivial Hecke correspondence $\tau'$ and $x\in V$ such that $x\in \tau'(x)$ and $V^{/x} \subset \tau'(V^{/x}) $. 
\end{proposition}

\begin{proof}
As $\phi_x \in U$, we have that $\tau_{h^n}(x)$ contains an $\F_q$-rational point for every $n$ by \autoref{thm: Heckefixed criterion}. Therefore, there is an increasing sequence of integers $n_i$ and a point $y\in V(\F_q)$ such that $y\in \tau_{h^{n_i}}(x)$ for every $n_i$. Therefore, we have that $x\in \tau_{h^{n_1}}^{\vee}\tau_{h^{n_i}}(x)$. For $i$ large enough, we have that every irreducible component of this correspondence is non-trivial. For some such $i$, let $\tau'$ be the irreducible component of $\tau_{h^{n_1}}^{\vee}\tau_{h^{n_i}}$ that satisfies $x\in \tau'(x)$. As $V$ has maximal monodromy, we have that $\tau'(V) = V$. The result now follows.

\end{proof}

To prove \autoref{thm:Tate-linear}, the idea is to use a rigidity result due to Chai \cite[Theorem]{Chairigidity} in conjunction with Proposition \ref{Heckefixedpoint}. However, the hypothesis required to apply Chai's theorem does not automatically hold in our setting and is rather non-trivial to check. The key then is to check this hypothesis; to address this problem, we work with Weyl-special points - a notion introduced by Chai and Oort in \cite[Section 5]{ChaiOortJacobian}.

\subsection{Weyl-special points}

Let $x\in \Sh(\F_q)$ be an ordinary point, which, via the Hodge embedding $\Sh\rightarrow \mc{A}_g$, is equivalently an abelian variety $A_x$ equipped with $G$-structure. Let $\tilde{x}\in \Ag(W(\mb{F}_q))$ denote the moduli point induced by the canonical lift of $A_x$. By \cite{Noot} (also proved using different methods in \cite{ShankarHecke}), we have that $\tilde{x} \in \Sh(W(\mb{F}_q))$
As the canonical lift is a CM abelian variety, the point $\tilde{x}$ is a special point, and is therefore induced by zero-dimensional Shimura variety $(T_x,h)$. We therefore have a map of $\Q$-groups $T_x\rightarrow G$. Note that elements $\alpha \in T_x(\Q)$ induce (quasi)-isogenies $A_x \rightarrow A_x$ that respect $G$-structure. Further, we have that $T_x$ contains $Z(G)$, the center of $G$.

We first recall from \cite[Defn. 5.3]{ChaiOortJacobian} the definition of a Weyl sub-torus of $G$. Note that for any field $k$ with algebraic closure $\bar{k}$,  and a  reductive group $\mscr{G}/k$, with a $k$-maximal torus $\mscr{T}$, there is a natural action of $\Gal_k$ on the character group $X^*(\mscr{T}_{\bar{k}})$.
\begin{defn}\label{defn: Weyl torus}
    A maximal $\Q$-torus $T\subset G$ is said to be a Weyl subtorus if the image of the natural action of $\Gal_{\Q}$ on the character lattice of $T_{\Qbar}$, contains the Weyl group $W(G_{\Qbar},T_{\Qbar})$. 
\end{defn}
We now define the notion of an ordinary Weyl special point:
\begin{defn}[Weyl-special points]\label{defn: Weyl-special}
    We say that an ordinary point $x\in \Sh(\F_q)$ is a Weyl special point if the torus $T_x \subset G$ is a Weyl subtorus. 
\end{defn}



The main result of this subsection is to show that ordinary Weyl special points are  abundant in $\Sh$. More specifically, we prove: 

\begin{theorem}\label{weylCMexists}
    Let $V\subset \Sh$ be as above. Then there exist infinitely many Weyl special points $x\in V^{\ord}$. Further, we may assume that $\phi_x \in U$ where $U$ is as in Proposition \ref{Heckefixedpoint}.  
\end{theorem}
We note that this recovers Chai--Oort's result which produces Weyl special points on Shimura varieties using different methods (Chai--Oort's argument uses Hilbert irreducibility -- a variant of Deligne's argument that produces special points on Shimura varieties). In \cite[Theorem 5.5]{ChaiOortJacobian}, Chai--Oort prove  that Weyl special points can be constructed to avoid finitely many sub-Shimura varieties of the ambient Shimura variety. Our method is also robust enough to recover \cite[Theorem 5.5]{ChaiOortJacobian}. However, as we do not need this stronger result, we will be content with merely proving Theorem \ref{weylCMexists}. 

The main property of Weyl tori which we will use to both prove Theorem \ref{weylCMexists} and to deduce the Tate-linearity of $V$ is that the adjoint of Weyl sub-tori do not contain non-trivial $\Q$-subtori (for example, see \cite[Proof of (5.8)]{ChaiOortJacobian}, and recall that we are assuming $G^{ad}$ is $\mb{Q}$-simple). 

We will need the following setup and preliminary propositions.
Let $T,T'$ be two split maximal tori inside a reductive group $G$, defined over an algebraically closed field. Then, there is an isomorphism between their Weyl groups $c: W(T) \rightarrow W(T')$, induced by the fact that $T$ and $T'$ are conjugate to each other. The isomorphism $c$ is only well defined up to conjugation by an element of $W(T)$, as it is induced by an isomorphism $T\rightarrow T'$ which is well defined only up to an element of $W(T)$. Therefore, an element $w'\in W(T')$ defines an element of $w\in W(T)$ up to conjugation: we say that the elements $w\in W(T)$ and $w'\in W(T')$ are \emph{conjugate} to each other.

\begin{proposition}\label{prop:localtorus}
     Let $\ell$ be a prime number, and  $G$  a split reductive group over $\F_{\ell}$. Suppose that $T'\subset G$ is a split maximal torus defined over $\F_{\ell}$ such that $G(\F_{\ell})$ contains a full set of coset representatives of $W(T')$. Then there is a maximal torus $T_w \subset G$ defined over $\F_{\ell}$ such that the Galois action of $\Gal(\overline{\F}_{\ell}/\F_{\ell})$ on the co-character lattice of $T_{w,\overline{\F}_{\ell}}$ contains an element of $W(T_{w,\overline{\F}_{\ell}})$ conjugate to $w \in W(T')$.
\end{proposition}

\begin{proof}
    We abuse notation and let $w \in G(\F_{\ell})$ denote a coset representative of $w\in W$. By Lang's theorem, there exists $\alpha \in G(\overline{\F}_{\ell})$ that satisfies $\alpha^{-1}\sigma(\alpha) = w$. We define  $T_w :=  \alpha T' \alpha^{-1}$: it is straightforward to check that it is a maximal torus of $G$ defined over $\F_{\ell}$. 

    We now compute the action of $\sigma$ on the co-character lattice of $T_w$. We may identify the co-character lattices $X_*(T')$ and $X_*(T_w)$ by conjugating by $\alpha$. Let $\mu: \G_m \rightarrow T'$ be some co-character. Then, $\sigma(\alpha \mu \alpha^{-1}) = \sigma(\alpha)\mu \sigma(\alpha^{-1}) = \alpha \alpha^{-1} \sigma(\alpha) \mu \sigma(\alpha)^{-1} \alpha \alpha^{-1} = \alpha (w\cdot \mu) \alpha^{-1}$. The proposition follows. 
    
    \end{proof}

We now give a  criterion that forces a torus to be a Weyl-torus. 
\begin{proposition}[Weyl-torus criterion]\label{prop: Weyl desiderata}
    Let $G$ be a reductive group defined over $\Z[1/N]$ and let $T\subset G$ be a maximal torus defined over $\Z[1/N]$. Let $S$ be the set of primes  not dividing $N$ at which $G$ splits. 

    Suppose that for every $w$ in the Weyl group of $G_{\bar{\Q}}$, there is a prime $\ell_w \in S$  
    and a maximal torus $T_w\subset G\otimes \mb{F}_{\ell_w}$ such that the Frobenius action on $X_*(T_{w, \overline{\mb{F}}_{\ell_w}})$ is conjugate to $w$.
    
    Suppose moreover that    $T \bmod \ell_w$ is conjugate over $\F_{\ell_w}$ to $T_w$. Then $T$ is a Weyl-torus of $G$. 
    
\end{proposition}
\begin{proof}
    Let $\rho: \Gal_{\Q} \rightarrow \GL(X_*(T_{\overline{\Q}}))$ be the representation of the Galois group on the co-character lattice of $T_{\Qbar}$, say with image $F$. By our hypothesis, $F$ intersects every conjugacy class of $W(T_{\overline{\Q}})$. Therefore, $F \cap W(T_{\overline{\Q}}) \subset W(T_{\overline{\Q}})$ intersects every conjugacy class; here we view $W(T_{\Qbar})$ as a subgroup of $\GL(X_*(T_{\overline{\Q}}))$. It is a well known fact about finite groups that this forces $F \cap W(T_{\overline{\Q}}) = W(T_{\overline{\Q}})$, i.e. $T$ is a Weyl torus, as required.
\end{proof}
We  make the following convenient
\begin{defn}
Let $S_{\spl}$ be the set of primes $\ell$ such that $G$ is split at $\ell$.
Note that $S_{\spl}$ is infinite by a straightforward application of Chebotarev density.    
\end{defn}
\begin{proposition}\label{prop: indep+cadoret}
    Let $Z/\mb{F}_q$ be a smooth, geometrically connected variety, $f: Z\rightarrow \Sh$ a map with big monodromy, and fix an integer $k\geq 1$. Then there exists $N=N(f)$ such that the following holds: for any $n\geq 1$, prime numbers $N<\ell_1< \cdots< \ell_n$ with $\ell_j\in S_{\spl}$  for each $j$, and elements $x_j\in \imag (G^{\SC}(\mb{Z}/\ell_n^k)\rightarrow G(\mb{Z}/\ell_n^k))$, there exists a  closed point $z\in Z$ whose Frobenius element is conjugate to $x_j$ in $G(\mb{Z}/\ell_n^k)$ for each $j$.
\end{proposition}

\begin{proof}

 By \cite[Thm. 1.2]{Bockle-Gajda-Petersen}, there exists a finite index subgroup $\Pi \leq \pi_1^{geo}(Z)$ such that the image of 
 \[
 \rho_{\infty}: \Pi \rightarrow \prod G(\mb{Z}_{\ell})
 \]
 is $\prod \rho_{\ell}(\Pi)$. 

 By \cite[Thm. 1.2]{cadoret2017geometric}, there exists $N$ such that $\rho_{\ell}(\Pi)=(\imag G^{\SC}(\Z_\ell) \rightarrow G(\mb{Z}_{\ell}))$ for all $\ell>N$. Therefore the representation
 \[
 \pi_1^{geo}(Z) \rightarrow \prod_{\ell>N} (\imag G^{\SC}(\Z_\ell) \rightarrow G(\mb{Z}_{\ell}))
 \]
is surjective. By Chebotarev density, we deduce  that this choice of $N$ suffices.

\end{proof}

We are now ready to prove that Weyl-special points are abundant in $V$. 

\begin{proof}[Proof of Theorem \ref{weylCMexists}]
    
    For each element $w\in W$, by an application of Chebotarev density, we can find  $\ell_w\in S_{\spl}$ satisfying  the hypothesis of Proposition \ref{prop:localtorus} for the reduction of $H^{\SC}$ mod $\ell_w$; let therefore  $T_w \subset H^{\SC}_{\F_{\ell_w}}$ be a maximal torus  output by Proposition \ref{prop:localtorus}. Let $y_w \in T_w(\F_{\ell_w})$ be a regular element. Let $x^{sc}_w \in G^{\SC}(\Z/\ell_w^2\Z)$ be a lift of $y_w$ 
    and finally let $x_w$ be the image of $x^{sc}_w$ in $G(\Z/\ell_w^2\Z)$.
    
    Further, as $y_w$ is a regular element, the centralizer of any element of $G^{\SC}(\Z_{\ell_w})$ that reduces to $x^{sc}_w$ must be a maximal torus, and therefore the mod $\ell_w$ centralizer must be $T_w$. 

    By \autoref{prop: indep+cadoret}, further increasing the $\ell_w$'s if necessary,  there exist infinitely many closed points $x\in V$ such that the action of Frobenius at $x$ on the $\ell_w$-adic Tate module of $A_x$ is conjugate to $x_w$ in $G(\Z/\ell_w^2 \Z)$, for each $w\in W$. We may also require that the action of Frobenius at $x$ is in $U$. Let $T_x \subset G$ be the $\mb{Q}$-torus defined by $A_x$. We have that the centralizer $Z_x$ of $T_x$ is a maximal torus of $G$, and is necessarily defined over $\Q$. By construction, the mod $\ell_w$ reduction of $Z_x$  contains a torus conjugate over $\F_{\ell_w}$ to $p^{sc}(T_w)$, where we recall that $p^{sc}$ denotes the map  $G^{\SC}\rightarrow H$ and, by  abuse of notation, any $\bmod \ \ell$ reduction of it. 
    
    By Proposition \ref{prop: Weyl desiderata}, we find that $Z_x$ is a Weyl torus.

    On the other hand,  we have that $T_x \subset Z_x$ is a $\Q$-rational subtorus of a Weyl subtorus. Further, we have that $T_x$ contains the center $Z(G)$. Therefore, we must have $T_x = Z(G)$ or $T_x = Z_x$. The former is impossible (as the Shimura datum induced by $T_x$ embedes in $G,X$, and so can't map to the center), and so the theorem follows. 
    
    
\end{proof}

\subsection{Tate-linearity of $V$}

We will first recall Chai's work. For an ordinary point $x\in \Sh(\mb{F}_q)$, recall that the formal completion  $\Sh^{\slash x}$ has the structure of a formal $p$-divisible group. Let  $\mathrm{Aut}(\Sh^{\slash x})$ denote the automorphism group of $\Sh^{\slash x}$ preserving the formal group structure on the latter. 

\begin{theorem}[Chai]\label{thm:chai-criterion}
    Suppose $\hat{V}\subset \Sh^{\slash x}$ is a formal subscheme stable under $\gamma \in \mathrm{Aut}(\Sh^{\slash x})$. Moreover, suppose that   $\gamma$ acts  with no invariants on $T_x\Sh^{\slash x}$. Then $\hat{V}$ is a formal subgroup of $\Sh^{\slash x}$.
\end{theorem}

We will now apply this to our setting and prove Theorem \ref{thm:Tate-linear}. 
\begin{proof}[Proof of \autoref{thm:Tate-linear}]
By  \autoref{weylCMexists}, there exists a Weyl special point $x\in V$ such that $x \in \tau(x)$ and $\tau_1(V^{/x}) \subset V^{/x}$ where $\tau_1$ is a local branch of $\tau$ at $x$. Let $\alpha \in \End(A_x)$ be the endomorphism induced by $\tau$ - we treat $\alpha$ as an element of $T_x(\Q)$. The assumption that $\tau_1(V^{\slash x}) \subset V^{\slash x}$ implies that $\alpha(V^{\slash x}) \subset V^{\slash x}$. Therefore, we have that $T'_x(\Z_p) \cdot V^{\slash x} \subset V^{\slash x}$, where $T'_x$ is the subtorus of $T_x$ generated by $\alpha$. As $\alpha$ is a $\Q$-point of $T_x$, we have that $T'_x$ is a subtorus defined over $\Q$. Further, the non-triviality of $\tau_1$ implies that the image of $T'_x$ in $G^{\ad}$ is non-trivial and therefore equals the image of $T_x$ as the latter is a Weyl-torus. Therefore, we may assume that $T'_x$ contains $T^{\der}_x := T_x \cap G^{\der}$.

By \autoref{thm:chai-criterion}, it suffices to show that the action of $T^{\der}_x$ does not have the trivial representation as a sub-representation. The action of $T^{\der}_x$ on the co-character lattice is induced by the conjugation action of $T^{\der}_x$ on the unipotent subgroup of $G$ induced by the Hodge-cocharacter associated with the canonical lift of $x$. The result follows as $T^{\der}_x$ is a maximal torus of $G^{\der}$.

\end{proof}

\subsection{Proof of \autoref{thm:maximmonodromy}}
We are now ready to prove our main theorem. The main inputs are Tate-linearity, big monodromy, D'Addezio's work on the parabolicity conjecture, and work of Chai (see also \cite{Pol}, \cite{MarcoPol}).

\begin{proof}
    The subvariety $V\subset \Sh$ has maximal prime-to-$p$ monodromy, and is Tate-linear by \autoref{thm:Tate-linear}. By \cite{companions}, the F-isocrystal $D[1/p]$ restricted to $V^{\ord}$ has geometric monodromy $G^{\der}$ in the overconvergent category, and therefore has monodromy $P\subset G^{\der}$ in the non-overconvergent category. Note that $D[1/p]|_{V^{ord}}$ has a natural slope filtration. By \cite[Theorem 1.1.1]{parabolicity}, $P$ is the parabolic induced by the slope filtration. Let $x\in V$ denote an ordinary point. By \cite{Chai} (also see \cite{Pol}), the smallest formal subtorus of $\Sh^{/x}$ containing $V^{/x}$ has rank equal to the dimension of the unipotent radical of $P$, and therefore 
is $\Sh^{/x}$ itself. The theorem now follows from the Tate-linearity of $V$. 
\end{proof}

\section{Function field formulation}
Let $F :=\F_p(t)$, $F_1 = \Fpbar(t)$, and let $\Fbar$ be an algebraic closure of $F$. As usual, let $(G, X)$ be a Shimura datum, $K$ a level subgroup which is hyperspecial at $p$, and $\ShimK$  denote the $\bmod \ p$ fiber of the associated Shimura variety. For brevity, we write $\Sh$ for $\ShimK$, and $\Sh_{\Fbar}$ its basechange to $\Fbar$. We say that a point $x\in \Shfbar(\Fbar)$ has big monodromy if it has a model over a finite extension of $F_1$ such that the associated Galois representation has big image.

\begin{theorem}\label{thm: function-field-version}
    Let $x\in \Shfbar(\Fbar)$ be a point with big monodromy, and $y_1, y_2, \cdots$ be an infinite sequence of distinct points isogenous to $x$ by prime-to-$p$ Hecke operators $\tau_1, \tau_2, \cdots$, respectively. Then the Zariski closure $V_{\Fbar}$ of $\{y_1, y_2, \cdots \}$ is a union of the connected components of  $\Shfbar$. 
\end{theorem}
With notation as in \autoref{thm: function-field-version}, for each $i$, let $V_{i,0, \Fbar}$ be the irreducible component of $\tau_i(V_{\Fbar})$ containing $x$.

\begin{proposition}\label{prop: v-big-monodromy}
    In the notation of \autoref{thm: function-field-version}, $V_{\Fbar}$ has big monodromy.
\end{proposition}

\begin{proof}
    There is a finite extension $\widetilde{F}/F_1$ such that $V_{\Fbar}$ and $y_1\in V_{\Fbar}(\Fbar)$ both descend to $\widetilde{F}$: let us call these $V_{\widetilde{F}}$ and $y_{1, \widetilde{F}}\in V_{\widetilde{F}}(\widetilde{F})$ respectively. Since $x$ has big monodromy by assumption, the Galois representation of $G_{\widetilde{F}}$ attached to $y_{1, \widetilde{F}}$ has big image. 

    Since $G$ is $\mb{Q}$-simple by assumption, the geometric monodromy of $V_{\widetilde{F}}$ either has the same Zariski closure as the monodromy of $V_{\widetilde{F}}$, or is trivial. The latter isn't possible as that would imply that $V$ is zero-dimensional, which contradicts Theorem \ref{thm: intro curve}. 
\end{proof}

In order to prove \autoref{thm: function-field-version}, it suffices to work in the following setting. 

\begin{proposition}\label{prop: key-function-field}
    \begin{enumerate}
    \item \label{item-big-monodromy}
    
    The field of definition of $V_{i,0, \Fbar}$ is bounded independently of $i$. Moreover, $V_{i,0, \Fbar}$ has big (geometric) monodromy with index bounded independently of $i$.

    \item \label{item-deg-bounded} $\deg(V_{i,0, \Fbar})$ is bounded independently of $i$;
   
    \end{enumerate} 
\end{proposition}
\begin{proof}
We first prove (\ref{item-big-monodromy}).    We may take a finite extension $E/F$ such that $\Sh_{\Fbar}, x, V_{\Fbar},$ and all Hecke correspondences are defined over $E$; in the following, we use the subscripts $E$ to refer to such descents to $E$, and similarly for finite extensions of $E$. Now, for each $i$, consider the Hecke correspondence $\tau_i^{\vee}$, which over $E$ is given by a diagram 
    \begin{equation}\label{diagram: hecke-function-field}
\begin{tikzcd}[column sep=small]
& S_{K_i}(G, X)_E=: \Sh_{K_i, E} \arrow[dl, "\pi_1"'] \arrow[dr, "\pi_h= \iota \circ \pi_2"] & \\
  S_K(G, X)_{E}=\Sh_{E}  &                         & S_K(G, X)_{E}=\Sh_{E}.
\end{tikzcd}
\end{equation}
As $V_{E}$ has big geometric monodromy, we claim  that there exists a finite extension $E_1/E$, independent of $i$, such that each irreducible component  of $\pi_1^{-1}(V_E)$ is defined over $E_1$. Indeed, there is a finite index $K'\subset K$ for which  there is a map 
\[
\iota: V_{\Fbar}\rightarrow S_{K'}(G, X)_{\Fbar}
\]
lifting the inclusion $V_{\Fbar}\subset \ShimK$, and 
whose image has maximal monodromy; let $E_1'$ be a field of definition of $\iota$ and  take $E_1$ to be the compositum of $E$ and $E_1'$. By \autoref{cor: can reduce level structure to get hecke irreducible}, there is a further finite index subgroup $K''\subset K$, with associated map 
\[
\pi_{K''}: S_{K''}(G, X)_{E_1}\rightarrow S_{K'}(G, X)_{E_1}
\]
such that for any $h\in G(\mb{A}^{f,p})$, we have that $\tau_h^{K''}(\pi_{K''}^{-1}(\iota(V_{\Fbar})))$ is irreducible, and hence defined over $E_1$; this implies that the original $\tau_h^{K}(V_{\Fbar})$ are also defined over $E_1$, as claimed.

Let $W_{i, E_1}$ be a geometrically irreducible component of $\pi_1^{-1}(V_E)$ which maps to $V_{i,0,\Fbar}$; this gives a descent of the latter to $E_1$, which we denote by $V_{i,0, E_1}$. This proves the first claim of (\ref{item-big-monodromy}).

To summarize, we obtain a commutative  diagram, defined over $E_1$: 
\[
\begin{tikzcd}
  W_{i, E_1} \arrow[r] \arrow[d]
    & \Sh_{K_i, E_1}\arrow[d, "\pi_h"] \\
  V_{i, 0, E_1} \arrow[r]
&  \Sh_{E_1}. 
\end{tikzcd}
\]
\begin{claim}\label{claim: big-deg-corr-function-field}
   Let $d=\deg \pi_h$. Then there exists a constant $\kappa_2>0$, independent of $i$, such that $\deg(\pi_h|_{W_{i, \Fbar}})\geq \kappa_2 d$.
\end{claim}
We first  show how to conclude the proof of (\ref{item-big-monodromy}), assuming this claim. Suppose that $\pi_h^{-1}(x)=\{x_1, \cdots, x_d\}$, so that we have a natural action of $\pi_1(\Sh_{\Fbar})$ on the latter, whose kernel is precisely $\pi_1(\Sh_{K_i, \Fbar})$. Wlog let $\{x_1, \cdots, x_{j}\}$ be the subset of those points contained in $W_{i, \Fbar}$, and we have a natural action of $\pi_1(V_{i,0, \Fbar})$ on this set, whose kernel is the image of $\pi_1(W_{i, \Fbar})$. We summarize this in the diagram
\[
\begin{tikzcd}
   \pi_1(W_{i, \Fbar}) \arrow[d]
    & \acts & \{x_1, \cdots, x_j\} \arrow[d, hook] \\
  \pi_1(V_{i, 0, \Fbar}) & \acts
&  \{x_1, \cdots, x_d\}. 
\end{tikzcd}
\]
Let $\kappa_1=[\pi_1(\Sh_{\Fbar}): \pi_1(V_{\Fbar})]$. Then 
\begin{align}
[\pi_1(\Sh_{\Fbar}): \pi_1(V_{i, 0, \Fbar})] 
&= \frac{[\pi_1(\Sh_{\Fbar}): \pi_1(W_{i, \Fbar})]}{j} \\
&\leq \kappa_1 \frac{[\pi_1(\Sh_{\Fbar}): \pi_1(\Sh_{K_i, \Fbar})]}{j} \\
& \leq \kappa_1\frac{d}{\kappa_2 d},
\end{align}
where the last inequality follows from \autoref{claim: big-deg-corr-function-field},
and so we are done.

\begin{proof}[Proof of \autoref{claim: big-deg-corr-function-field}]
    For a group $H$ acting on a finite set of size $d$, and a constant $\kappa>0$, we say that the action  is $\kappa$-big if every orbit has size $\geq \kappa d$. 
    
    In this language, it suffices to show that the action of $\pi_1(V_{i, 0, \Fbar})$ on $\{x_1, \cdots, x_d\}$ is $\kappa_2$-big, for some $\kappa_2>0$ independent of $i$, as $\{x_1, \cdots, x_j\}$ is a $\pi_1(V_{i, 0, \Fbar})$-orbit.

The action of $\Gal_{E_1}$ on $\{x_1, \cdots, x_d\}$ is manifestly $\kappa_2$-big for some $\kappa_2$ independent of $i$, since $x$ was assumed to have big monodromy. Since each geometric irreducible component of $\widetilde{V_{i, 0, E_1}}$ is defined over $E_1$, each $\pi_1(V_{i, 0, \Fbar})$-orbit  on $\{x_1, \cdots, x_d\}$ is a union of $\Gal_{E_1}$-orbits, so we are done.

\end{proof}

This concludes the proof of (\ref{item-big-monodromy}). The proof of (\ref{item-deg-bounded}) is identical to that given in the proof of \autoref{prop:key}.
\end{proof}

\begin{proof}[Proof of \autoref{thm: function-field-version}]


    Spreading out, we may find 
    \begin{itemize}
        \item a smooth, geometrically connected curve $C/\mb{F}_q$, with an embedding $\iota:  \mb{F}_q(C)\xhookrightarrow{} \Fbar$, and a map 
        \[
        x_C: C\rightarrow \Sh
        \]
        whose generic point, composed with  $\iota$, recovers $x$, 
        \item  $V_{C}$ an irreducible subscheme of $\Sh \times C$ whose basechange to $\Shfbar$ recovers $V$,
        \item for each $i$, a subscheme  $V_{i,0, C}\subset \Sh \times C$ which is a component of $\tau_i(V_C)$, and whose basechange to $\Shfbar$ recovers $V_{i,0}$. 
    
    \end{itemize}

By replacing $K$ with a finite index subgroup $K'$, we may assume that $V$ has maximal geometric monodromy.
By \autoref{prop: key-function-field}(\ref{item-deg-bounded}), the degrees of $V_{i,0}$ are absolutely bounded, and hence the same is true of degrees of $V_{i,0, c}$ for any closed point $c\in C$. Here, define  $V_{i,0, c}$ via the Cartesian square: 
\[
\begin{tikzcd}
  V_{i,0, c} \arrow[r] \arrow[d]
    & V_{i,0, C} \arrow[d] \\
  \Sh \times c \arrow[r]
&  \Sh \times C. 
\end{tikzcd}
\]

Increasing  $q$ if necessary, we may assume that $C$ has a point $c\in C(\mb{F}_q)$. Further, by Proposition \ref{prop: v-big-monodromy}, we have that $V_{c}$ also has maximal monodromy. Then, as follows from the corresponding statements for the subschemes $V_{i,0,C}$'s, the $V_{i,0,c}$ are defined over $\mb{F}_q$ and have bounded degrees. Hence we find $i, j$ such that $V_{i,0,c}=V_{j,0,c}$. As in the setting of \autoref{thm: intro curve}, we have a sequence of non-trivial Hecke correspondences which fix the subvariety $V_c$. We now directly apply \autoref{thm:Tate-linear} to conclude.

\end{proof}

\bibliographystyle{alpha}
\bibliography{biblio}

\begin{thebibliography}{COU01}

\bibitem[And89]{Andre}
Yves Andr\'e.
\newblock {\em {$G$}-functions and geometry}, volume E13 of {\em Aspects of Mathematics}.
\newblock Friedr. Vieweg \& Sohn, Braunschweig, 1989.

\bibitem[BGP18]{Bockle-Gajda-Petersen}
Gebhard B{\"o}ckle, Wojciech Gajda, and Sebastian Petersen.
\newblock Independence of $\ell$-adic representations of geometric {G}alois groups.
\newblock {\em Journal f{\"u}r die Reine und Angewandte Mathematik}, 2018(736), 2018.

\bibitem[Cha]{Chai}
Ching-Li Chai.
\newblock Families of ordinary abelian varieties: canonical coordinates, p-adic monodromy, tate-linear subvarieties and hecke orbits.
\newblock {\em Preprint. \url{https://www2.math.upenn.edu/~chai/papers_pdf/fam_ord_av.pdf}}.

\bibitem[Cha95]{Chai-inventiones}
Ching-Li Chai.
\newblock Every ordinary symplectic isogeny class in positive characteristic is dense in the moduli.
\newblock {\em Invent. Math.}, 121(3):439--479, 1995.

\bibitem[Cha08]{Chairigidity}
Ching-Li Chai.
\newblock A rigidity result for {$p$}-divisible formal groups.
\newblock {\em Asian J. Math.}, 12(2):193--202, 2008.

\bibitem[CHT17]{cadoret2017geometric}
Anna Cadoret, Chun-Yin Hui, and Akio Tamagawa.
\newblock Geometric monodromy—semisimplicity and maximality.
\newblock {\em Annals of Mathematics}, 186(1):205--236, 2017.

\bibitem[CO12]{ChaiOortJacobian}
Ching-Li Chai and Frans Oort.
\newblock Abelian varieties isogenous to a {Ja}cobian.
\newblock {\em Annals of Mathematics}, pages 589--635, 2012.

\bibitem[COU01]{ClozelOhUllmo}
Laurent Clozel, Hee Oh, and Emmanuel Ullmo.
\newblock Hecke operators and equidistribution of {H}ecke points.
\newblock {\em Invent. Math.}, 144(2):327--351, 2001.

\bibitem[D'A23]{parabolicity}
Marco D'Addezio.
\newblock Parabolicity conjecture of {{\(F\)}}-isocrystals.
\newblock {\em Ann. Math. (2)}, 198(2):619--656, 2023.

\bibitem[dGF24]{Olivier}
Olivier de~Gaay~Fortman.
\newblock Abelian varieties with no power isogenous to a jacobian.
\newblock {\em Compositio Math.}, 2024.

\bibitem[DvH22]{MarcoPol}
Marco D'Addezio and Pol van Hoften.
\newblock Hecke orbits on {S}himura varieties of {H}odge type.
\newblock {\em arXiv:2205.10344}, 2022.

\bibitem[Kis10]{Kisinintegral}
Mark Kisin.
\newblock Integral models for {S}himura varieties of abelian type.
\newblock {\em J. Amer. Math. Soc.}, 23(4):967--1012, 2010.

\bibitem[Kol13]{kollar2013rational}
J{\'a}nos Koll{\'a}r.
\newblock {\em Rational curves on algebraic varieties}, volume~32.
\newblock Springer Science \& Business Media, 2013.

\bibitem[Lan]{Kai-Wen}
Kai-Wen. Lan.
\newblock An example-based introduction to shimura varieties.
\newblock {\em {P}roceedings of the ETHZ Summer School on Motives and Complex Multiplication.}

\bibitem[Mil]{Milneintro}
James Milne.
\newblock Introduction to shimura varieties.
\newblock {\em Preprint. \url{https://www.jmilne.org/math/xnotes/svi.pdf}}.

\bibitem[Moo98]{Moonenlinearity}
Ben Moonen.
\newblock Linearity properties of {S}himura varieties. {I}.
\newblock {\em J. Algebraic Geom.}, 7(3):539--567, 1998.

\bibitem[MST22]{MST}
Davesh Maulik, Ananth~N. Shankar, and Yunqing Tang.
\newblock Picard ranks of {K}3 surfaces over function fields and the {H}ecke orbit conjecture.
\newblock {\em Invent. Math.}, 228(3):1075--1143, 2022.

\bibitem[MZ20]{MasserZannier}
David Masser and Umberto Zannier.
\newblock Abelian varieties isogenous to no {J}acobian.
\newblock {\em Ann. of Math. (2)}, 191(2):635--674, 2020.

\bibitem[Noo96]{Noot}
Rutger Noot.
\newblock Models of {S}himura varieties in mixed characteristic.
\newblock {\em J. Algebraic Geom.}, 5(1):187--207, 1996.

\bibitem[Orr15]{orr}
Martin Orr.
\newblock Families of abelian varieties with many isogenous fibres.
\newblock {\em J. Reine Angew. Math.}, 705:211--231, 2015.

\bibitem[P\'22]{companions}
Ambrus P\'al.
\newblock The {$p$}-adic monodromy group of abelian varieties over global function fields of characteristic {$p$}.
\newblock {\em Doc. Math.}, 27:1509--1579, 2022.

\bibitem[Per19]{keerthi}
Keerthi~Madapusi Pera.
\newblock Toroidal compactifications of integral models of {Shimura} varieties of {Hodge} type.
\newblock {\em Ann. Sci. {\'E}c. Norm. Sup{\'e}r. (4)}, 52(2):393--514, 2019.

\bibitem[Pin05]{Pink}
Richard Pink.
\newblock A combination of the conjectures of {M}ordell-{L}ang and {A}ndr\'e-{O}ort.
\newblock In {\em Geometric methods in algebra and number theory}, volume 235 of {\em Progr. Math.}, pages 251--282. Birkh\"auser Boston, Boston, MA, 2005.

\bibitem[RY24]{richard-yafaev-height}
Rodolphe Richard and Andrei Yafaev.
\newblock Height functions on {Hecke} orbits and the generalised {Andr{\'e}}-{Pink}-{Zannier} conjecture.
\newblock {\em Compos. Math.}, 160(11):2531--2584, 2024.

\bibitem[RY25]{AndrePinkAbelian}
Rodolphe Richard and Andrei Yafaev.
\newblock Generalised {A}ndr{\'e}-{P}ink-{Z}annier conjecture for {S}himura varieties of abelian type.
\newblock {\em Publications math{\'e}matiques de l'IH{\'E}S}, pages 1--83, 2025.

\bibitem[Sha17]{ShankarHecke}
Ananth~N. Shankar.
\newblock The {H}ecke orbit conjecture for "modèles étranges".
\newblock {\em Preprint.}, 2017.

\bibitem[ST18]{AJunlikely}
Ananth~N. Shankar and Jacob Tsimerman.
\newblock Unlikely intersections in finite characteristic.
\newblock {\em Forum Math. Sigma}, 6:Paper No. e13, 17, 2018.

\bibitem[ST25]{AJ}
Ananth~N. Shankar and Jacob Tsimerman.
\newblock Abelian varieties not isogenous to {J}acobians over global fields.
\newblock {\em Duke Mathematical Journal}, 2025.

\bibitem[Tit79]{tits-reductive-groups}
Jacques Tits.
\newblock Reductive groups over local fields.
\newblock Automorphic forms, representations and {L}-functions, {Proc}. {Symp}. {Pure} {Math}. {Am}. {Math}. {Soc}., {Corvallis}/{Oregon} 1977, {Proc}. {Symp}. {Pure} {Math}. 33, 1, 29-69 (1979)., 1979.

\bibitem[Tsi12]{Jacob}
Jacob Tsimerman.
\newblock The existence of an abelian variety over {$\overline{\Bbb Q}$} isogenous to no {J}acobian.
\newblock {\em Ann. of Math. (2)}, 176(1):637--650, 2012.

\bibitem[vH24]{Pol}
Pol van Hoften.
\newblock On the ordinary {H}ecke orbit conjecture.
\newblock {\em Algebra Number Theory}, 18(5):847--898, 2024.

\bibitem[Yaf00]{yafaev2000sous}
Andrei Yafaev.
\newblock Sous-vari{\'e}t{\'e}s des vari{\'e}t{\'e}s de {S}himura.
\newblock {\em These de Doctorat, Universit{\'e} de Rennes}, 1, 2000.

\end{thebibliography}

\end{document}